\documentclass[12pt]{article}

\usepackage{amsmath,amsthm,amscd,amssymb,latexsym,upref,stmaryrd,epsfig,graphicx}
\usepackage[cp1251]{inputenc}
\usepackage[english]{babel}
\usepackage{mathtext}
\usepackage{amsfonts}
\usepackage {tikz}
\usepackage{hyperref}
\usepackage{cite}
\usepackage[numbers,sort&compress]{natbib}
\newcommand*{\mailto}[1]{\href{mailto:#1}{\nolinkurl{#1}}}
\usepackage{amsmath}
\usepackage{amsthm}

\newtheorem{theorem}{Theorem}[section]
\newtheorem{lemma}{Lemma}[section]

\numberwithin{equation}{section}

\begin{document}
\thispagestyle{empty}
 {\Large\bf     Some Bohr-type inequalities for  several subclasses of}\\\\
{\centerline {\Large\bf  harmonic functions} }\\\\
\begin{center}

{\bf Jianying Zhou  \quad  Wanqing Hou  \quad Boyong Long$^*$}
\let\thefootnote\relax\footnotetext{\\This work is supported by NSFC (No.12271001) and Natural Science Foundation of Anhui Province
(2308085MA03), China. \\$^*$Corresponding author.\\ \quad Email: jianying820@163.com;\quad wanqinghou99@163.com; \quad boyonglong@163.com
}
\end{center}

\centerline {\small  (School of Mathematical Sciences, Anhui University, Hefei 230601, China)}
 \hspace*{\fill} \\\\
\noindent{\bf Abstract:}
{In this article,  Bohr type inequalities for some complex valued harmonic functions defined on the unit disk are given. All the results are sharp.
}

\medskip


\section{Introduction}

Let $\mathbb{D}=\left\{z:|z|<1\right\}$ denote the open unit disk. The classic Bohr's theorem\cite{bohr1914theorem} states that if $f(z)=\sum_{n=0}^{\infty}a_n z^n$ is a bounded analytic function defined on  $\mathbb{D}$ and $|f(z)|<1$ for all $z\in\mathbb{D}$, then
\begin{align}\label{1.1}
 \sum_{n=0}^{\infty}|a_n||z|^n<1
\end{align}
holds for $|z|=r\leq 1/3$ and $1/3$ is the best possible. Inequality (\ref{1.1}) is so-called Bohr's inequality.

In recent years, many scholars have done a lot of work in  Bohr's inequalities of analytic functions. For example, Allu and Halder\cite{allu2021halder} have studied Bohr's inequalities for close-to-convex analytic functions;  Hu\cite{hu2022bohr} has considered Bohr's inequalities for bounded analytic functions of schwarz functions; Allu and Halder\cite{allu2021starlike} have discussed Bohr's inequalities for Ma-Minda functions.

Initially, the definition domain and image domain of the functions involved in Bohr inequality are restricted to the unit disk. Later, this situation has changed. Evdoridis,  Ponnusamy and Rasila \cite{evdoridis2021improved} discussed the Bohr's inequality for functions defined on shifted disk rather than unit disk. \cite{abu2017bohr,ali2017note} have studied bounded analytic functions from $\mathbb{D}$ to punctured disk and wedge domain, respectively.  Furthermore, the classic Bohr inequality has been changed into other forms: Liu\cite{Ming2018Bohr} has replaced the power series coefficients with higher derivatives; Kayumov and Ponnusamy\cite{kayumov2017bohr} have added the positive term to the left-hand side of the classic inequality; Wu\cite{wule2022} has explored some Bohr-type inequalities with one parameter or involving convex combination. It is worth mentioning that the right side of the inequality(\ref{1.1}) was also generalized. For analytic functions $f$ defined on the $\mathbb{D}$, let $d(f(0),\partial f(\mathbb{D}))$ replace the $1-|a_0|$, where $d(f(0),\partial f(\mathbb{D}))$ denotes the distance between $f(0)$ and $\partial f(\mathbb{D})$. Then Bohr's inequality(\ref{1.1}) has been transformed into
\begin{align}\label{1.2}
 \sum_{n=1}^{\infty}|a_n||z|^n\leq d(f(0),\partial f(\mathbb{D})).
\end{align}

Let $\mathcal{H}$ be the class of all complex-valued harmonic functions $f=h+\overline{g}$ defined on the $\mathbb{D}$, where $h$ and $g$ are analytic in $\mathbb{D}$ with the normalization $h(0)=h'(0)-1=0$ and $g(0)=0$. Let $\mathcal{H}_{0}=\left\{{f=h+\overline{g}\in\mathcal{H}: g'(0)=0}\right\}$. Then, each $f=h+\overline{g}\in\mathcal{H}_{0}$ has the following form
$$h(z)=z+\sum_{n=2}^{\infty}a_{n}z^{n} \quad and \quad g(z)=\sum_{n=2}^{\infty}b_{n}z^{n}.$$
In recent decades,  harmonic function has become  a research hotspot.  Iwaniec, Kovalev and Onninen\cite{iwaniec2011nitsche} discussed the Nitsche conjecture of harmonic mappings. Chuaqui, Duren and Osgood\cite{chuaqui2003schwarzian} explored the Schwarzian derivative for harmonic mappings. Furthermore, Axler, Bourdon and Wade\cite{axler2013harmonic} introduced harmonic function theory of higher dimensions.
Schoen and Yau\cite{schoen1978univalent} demonstrated univalent harmonic maps between surfaces. Iwaniec and Onninen\cite{iwaniec2019rado} established $Rad\acute{o}$-Kneser-Choquet Theorem in p-harmonic setting for simply connected domains.

The functions involved in the initial Bohr inequality were analytical functions. Later, Bohr inequalities and Bohr phenomena of harmonic functions were also studied. Huang, Liu and Ponnusamy\cite{huang2021bohr} discussed Bohr type inequalities for harmonic mappings with a multiple zero at the origin; Ali, Abdulhadi and Ng\cite{ali2016bohr}, Gangania and Kumar\cite{gangania2022bohr}, Allu and Halder\cite{allu2021bohr} introduced the Bohr radii for starlike logharmonic mappings, univalent harmonic mappings, some certain harmonic functions, respectively. Also, Liu, Ponnusamy and Wang\cite{liu2020bohr} considered Bohr's phenomenon for the classes of K-quasiregular harmonic mappings. What's more, if the analytic part of the harmonic function satisfies some subordination relations, Liu and Ponnusamy\cite{liu2019bohr} have studied its Bohr's phenomenon.

There are several classes of harmonic functions in $\mathbb{D}$.

$$\mathcal{P_{H}}=\{f=h+\overline{g}\in\mathcal{H}: Re(h'(z))>|g(z)|, z\in\mathbb{D}\};$$
$$\mathcal{P}^{0}_{H}(\alpha)=\{f=h+\overline{g}\in\mathcal{H}_{0}: Re(h'(z)-\alpha)>|g(z)|, 0 \leq \alpha<1, z\in\mathbb{D}\};$$
$$\widetilde{G^{0}_{H}}(\beta)=\{f=h+\overline{g}\in\mathcal{H}_{0}: Re(\frac{h(z)}{z}-\beta)>|\frac{g(z)}{z}|,0 \leq \beta<1, z\in\mathbb{D}\};$$
$$\mathcal{W}^{0}_{H}(\alpha)=\{f=h+\overline{g}\in\mathcal{H}_{0}: Re(h'(z)+\alpha zh''(z))>|(g'(z)+\alpha zg''(z))|,0\leq\alpha<1, z\in\mathbb{D}\};$$
$$ \mathcal{G}^{k}_{H}(\alpha)=\{f=h+\overline{g}\in\mathcal{H}^{0}_{k}: Re(1-\alpha)\frac{h(z)}{z}+\alpha h'(z)>|(1-\alpha)\frac{g(z)}{z}+\alpha g'(z)|, \alpha\geq0, k\geq1, z\in\mathbb{D}\},$$
where $$\mathcal{H}^{0}_{k}=\{f=h+\overline{g}\in\mathcal{H}: h'(0)-1=g'(0)=h''(0)=...=h^{(k)}(0)=g^{(k)}(0)=0\}, \mathcal{H}^{0}_{1}=\mathcal{H}_{0}.$$

In this paper, we would like to generalize the following two theorems to some classes of harmonic functions.\\

\noindent{\bf Theorem A} \cite{kayumov2017bohr}
Suppose that $f(z)$ is analytic and $|f(z)|<1$ in $\mathbb{D}$. Then
\begin{equation*}
|f(z)|+ \sum_{n = N}^\infty |a_n||z|^n \leq 1
\end{equation*}
holds for $|z|=r\leq r_N$, where $r_N$ is the unique positive root of the equation $$2(1+r)r_N-(1-r)^2=0$$ and the radius $r_N$ is the best possible.

\noindent{\bf Theorem B}\cite{kayumov2018improved}
Suppose that $f(z)=\sum_{n=0}^{\infty}a_{n}z^{n}$ is analytic in $\mathbb{D}$, $|f(z)|<1$ and $S_{r}$ denotes the area of the image of the subdisk $|z|<r$ under the mapping $f$. Then
$$\sum_{n=0}^{\infty}|a_{n}||z|^{n} +\frac{16}{9}(\frac{S_{r}}{\pi}) \leq 1 \quad for \quad |z|\leq \frac{1}{3}$$
and the numbers $1/3$ and $16/9$ cannot be improved.

\section{ Some Lemmas }
In order to establish our main results, we need the following lemmas.

\begin{lemma}\label{lemma2.1}\rm{\cite{ghosh2019subclass}
Suppose that $f(z) = h + \bar{g}\in W^{0}_{H}(\alpha)$ for $\alpha \geq 0$. Then for $n\geq2$,
\begin{flalign*}
 \begin{split}
(i)\,\,&|a_n|+|b_n|\leq \frac{2}{\alpha n^2+n(1-\alpha)};\\
(ii)\,\,& ||a_n|-|b_n||\leq \frac{2}{\alpha n^2+n(1-\alpha)};\\
(iii) \,\,&  |a_n| \leq \frac{2}{\alpha n^2+n(1-\alpha)}.
 \end{split}&
\end{flalign*}
All these results are sharp for the function $f$ be given by $f(z)=z+2\sum\limits_{n = 2}^\infty \frac{1}{\alpha n^2+n(1-\alpha) }z^n$. }
\end{lemma}

\begin{lemma}\label{lemma2.2}\rm{\cite{ghosh2019subclass}
Suppose that $f(z) = h + \bar{g}\in W^{0}_{H}(\alpha)$ for $0\leq\alpha < 1$. Then
\begin{equation*}
r + 2\sum\limits_{n = 2}^\infty \frac{(-1)^{n-1} r^n}{\alpha n^2+n(1-\alpha)} \leq |f(z)|\leq r + 2\sum\limits_{n = 2}^\infty \frac{ r^n}{\alpha n^2+n(1-\alpha)}
\end{equation*}holds for $r=|z|$, and the equality holds for  $f(z)=z+2\sum\limits_{n = 2}^\infty \frac{1}{\alpha n^2+n(1-\alpha) }z^n$ or its rotations.}
\end{lemma}

\begin{lemma}\label{lemma2.3}\rm{\cite{liu2019geometric}
Suppose that $f(z) = h + \bar{g}\in \mathcal{G}^{k}_{H}(\alpha)$ for $\alpha \geq 0$ and $k\geq1$ . Then for $n\geq k+1$,
\begin{flalign*}
 \begin{split}
(i)\,\,&|a_n|+|b_n|\leq \frac{2}{1+(n-1)\alpha};\\
(ii)\,\,& ||a_n|-|b_n||\leq \frac{2}{1+(n-1)\alpha};\\
(iii) \,\,&  |a_n| \leq \frac{2}{1+(n-1)\alpha}.
 \end{split}&
\end{flalign*}
All these results are sharp for the function $f$ be given by $f_{k}(z)=z+2\sum\limits_{N\in I} \frac{1}{1+(N-1)\alpha }z^{N}$, $I=\left\{N|N=kj+1, j=1,2,\cdot\cdot\cdot\right\}$. }
\end{lemma}

\begin{lemma}\label{lemma2.4}\rm{\cite{liu2019geometric}
Suppose that $f(z) = h + \bar{g}\in \mathcal{G}^{k}_{H}(\alpha)$ for $\alpha \geq 0$ and $k\geq1$ . Then
\begin{equation*}
r + 2\sum\limits_{j = 1}^\infty \frac{(-1)^{j} r^{kj+1}}{1+kj\alpha} \leq |f(z)|\leq r + 2\sum\limits_{j = 1}^\infty \frac{ r^{kj+1}}{1+kj\alpha}
\end{equation*} holds for $r=|z|$, and the equality holds for  $f_{k}(z)=z+2\sum\limits_{N\in I} \frac{1}{1+(N-1)\alpha }z^{N}$ or its rotations, where $I=\left\{N|N=kj+1, j=1,2,\cdot\cdot\cdot\right\}$.}

\begin{lemma}\label{lemma2.5}
\rm { Suppose $f(z)\in \widetilde{G^{0}_{H}}(\beta)$ for $0\leq \beta<1$. Then
\begin{equation*}
\beta |z|+(1-\beta)\frac{1-|z|}{1+|z|}|z|\leq |f(z)|\leq \beta |z|+(1-\beta)\frac{1+|z|}{1-|z|}|z|.
\end{equation*}} Both inequalities are sharp for $f(z)=z+2(1-\beta)\sum\limits_{n = 2}^\infty z^n$ or its rotations.
\end{lemma}

\begin{proof}
According to the assumption, we have
\begin{align*}
  Re\left(\frac{h(z)}{z} - \beta\right) >\left|\frac{g(z)}{z}\right|.
\end{align*}
Then
\begin{align*}
  Re\left(\frac{h(z)}{z}+\epsilon \frac{g(z)}{z}- \beta\right) >0
\end{align*}
holds for $\forall \, |\epsilon|=1 $ and $\forall \, z\in\mathbb{D}$.

Let $$ G(z)=\frac{\frac{h(z)}{z}+\epsilon \frac{g(z)}{z}- \beta}{1-\beta},$$  $$ R(z)=\frac{1+z}{1-z}.$$
It is obvious that $G(z)$ is an analytic function on $\mathbb{D}$ and $G(0)=1$, $Re(G(z))>0$.  So it holds  that $ G(z)\prec R(z)$. Then there exists an analytic function $\omega(z)$, $\omega(0)=0$, such that
\begin{align*}
 \frac{\frac{h(z)}{z}+\epsilon \frac{g(z)}{z}- \beta}{1-\beta}=\frac{1+\omega(z)}{1-\omega(z)}.
\end{align*}
 Equivalently,
 \begin{align*}
 \frac{h(z)}{z}+\epsilon \frac{g(z)}{z}=\beta +(1-\beta)\frac{1+\omega(z)}{1-\omega(z)}.
\end{align*}

Let $$F(z)=h(z)+\epsilon g(z).$$ Then
\begin{align*}
 \frac{F(z)}{z}=\frac{h(z)}{z}+\epsilon \frac{g(z)}{z}=\beta +(1-\beta)\frac{1+\omega(z)}{1-\omega(z)}.
\end{align*}
It follows that
\begin{align*}
 \left|\frac{F(z)}{z}\right|=&|\beta +(1-\beta)\frac{1+\omega(z)}{1-\omega(z)}|\leq \beta +(1-\beta)(1+2\sum\limits_{n = 1}^\infty  |\omega(z)|^n)\\
 \leq& \beta +(1-\beta)(1+2\sum\limits_{n = 1}^\infty  |z|^n)=\beta +(1-\beta)\frac{1+|z|}{1-|z|}.
\end{align*}
Similarly,
\begin{align*}
 \left|\frac{F(z)}{z}\right|\geq& Re(\beta +(1-\beta)\frac{1+\omega(z)}{1-\omega(z)})= \beta +(1-\beta)Re(1+2\sum\limits_{n = 1}^\infty  |\omega(z)|^n)\\
 \geq& \beta +(1-\beta)(1+2\sum\limits_{n = 1}^\infty (- |z|)^n)=\beta +(1-\beta)\frac{1-|z|}{1+|z|}.
\end{align*}
Then we have
\begin{align*}
  \beta |z|+(1-\beta)\frac{1-|z|}{1+|z|}|z|\leq |F(z)|\leq \beta |z|+(1-\beta)\frac{1+|z|}{1-|z|}|z|,\\
\end{align*}
and the arbitrariness of $\varepsilon$ along with $F(z)=h(z)+\epsilon g(z)$ leads to
\begin{align*}
  \beta |z|+(1-\beta)\frac{1-|z|}{1+|z|}|z|\leq |f(z)|\leq \beta |z|+(1-\beta)\frac{1+|z|}{1-|z|}|z|.
\end{align*}
\end{proof}

\end{lemma}

\section{ Main results }

\begin{theorem}\label{theorem3.1}
\rm { Suppose $f(z)\in \widetilde{G^{0}_{H}}(\beta)$ for $0 < \beta< 1$, it holds that
\begin{equation*}
|f(z)|+ |z|+\sum_{n=2}^{\infty}(|a_{n}|+|b_n|)|z|^{n} \leq d(f(0),\partial f(\mathbb{D}))
\end{equation*} for $|z|=r\leq r_1$,
where  $r_1$ is the unique root of equation $$(2-4\beta)r^2+(2+\beta)r-1=0$$ in the interval $(0,1)$. The radius $r_1$ is the best possible.}
\end{theorem}

\begin{proof}
According to the assumption,  $f(z)\in \widetilde{G^{0}_{H}}(\beta)$, we obtain
\begin{align*}
Re\left(\frac{h(z)}{z}-\beta\right)>\left|\frac{g(z)}{z}\right|.
\end{align*}
It follows that
\begin{align*}
Re\left(\frac{h(z)}{z}+\varepsilon \frac{g(z)}{z}-\beta\right)>0
\end{align*}
holds for $\forall \,|\epsilon|=1 $ and $\forall \,z\in\mathbb{D}$. So that there exists an analytic function $p(z)=1+\sum_{n=1}^{\infty}p_n z^n$, $Re(p(z))>0$,
such that
\begin{align}\label{3.1}
\frac{h(z)}{z}+\varepsilon \frac{g(z)}{z}-\beta=(1-\beta)p(z).
\end{align}
Comparing the  coefficients of both sides of the  inequality (\ref{3.1}), we obtain that
\begin{align}\label{3.2}
 a_n+\varepsilon b_n=(1-\beta)p_{n-1}
\end{align} for  $n\geq2$.
By the $Carath\acute{e}odory$ $ Lemma$, $|p_n|\leq 2$ holds for $n\geq1$. By the arbitrariness of $\varepsilon$, we have
 \begin{align}\label{3.3}
 |a_n|+|b_n|\leq 2(1-\beta).
 \end{align}

Using  inequality (\ref{3.3})and Lemma \ref{lemma2.5}, we have
\begin{align*}
|f(z)|+|z|+\sum_{n=2}^{\infty}(|a_{n}|+|b_n|)|z|^{n}
\leq& \beta r+ (1-\beta)(\frac{1+r}{1-r})r+r+2(1-\beta)\sum\limits_{n = 2}^\infty r^n\\
=&\frac{(2-4\beta)r^2+2r}{1-r}.
\end{align*}
Noting  that $f(0)=0$, and using the Lemma \ref{lemma2.5}, we have
\begin{align*}
 d(f(0),\partial f(\mathbb{D})) =\liminf_{|z|\rightarrow1}|f(z)-f(0)|\geq \beta.
\end{align*}
Now, we need to show that $$\frac{(2-4\beta)r^2+2r}{1-r}\leq \beta$$ holds for $r\leq r_{1}$. It is equivalent to show that $(2-4\beta)r^2+(2+\beta)r-\beta\leq0 $ holds for $r\leq r_{1}$.

Let $$Q_1 (r)=(2-4\beta)r^2+(2+\beta)r-\beta.$$
Observe that $Q_1 (0)=-\beta<0$ and $Q_1 (1)=4-4\beta>0$. Thus, by the continuity and monotonicity of $Q_1 (r)$, there exists a unique root $r_1\in(0,1)$ such that $Q_1 (r_1)=0$. When
 $ r\leq r_1$, $Q_1 (r)\leq 0$ holds, then
 \begin{align*}
|f(z)|+|z|+\sum_{n=2}^{\infty}(|a_{n}|+|b_n|)|z|^{n}
\leq \frac{(2-4\beta)r^2+2r}{1-r} \leq \beta.
\end{align*}

To show the sharpness of the number $r_1$, we consider the function
\begin{equation*}
f(z) =z+ 2(1-\beta)\sum_{n = 2}^\infty z^n.
\end{equation*}
 Taking $|z|=r$, we have
\begin{align*}
|f(z)|+|z|+\sum_{n=2}^{\infty}(|a_{n}|+|b_n|)|z|^{n}=&2r+4(1-\beta)\sum_{n=2}^{\infty}r^n\\
=& \frac{(2-4\beta)r^2+2r}{1-r}.
\end{align*}
According to the above proof,  when $r=r_1$, $$\frac{(2-4\beta)r^2+2r}{1-r}=\beta=d(f(0),\partial f(\mathbb{D})).$$\\
This proves the sharpness of $ r_1$.
\end{proof}

\begin{theorem}\rm
\label{theorem3.2} Suppose that $f=h+\overline{g}\in \widetilde{G^{0}_{H}}(\beta)$ for $0< \beta<1$. Let $S_{r}$ denotes the area of the image of the subdisk $|z|<r$ under the mapping $f$. Then
\begin{equation*}
|z|+\sum_{n=2}^{\infty}(|a_{n}|+|b_{n}|)|z|^{n}+  \frac{S_{r}}{\pi}  \leq d(f(0),\partial f(\mathbb{D}))
\end{equation*} holds for $|z|=r\leq r_{2}$,
where $r_{2}$ is the unique  root of the equation $$r+2(1-\beta)\frac{r^{2}}{1-r}+r^{2}+4(1-\beta)^{2}\frac{(2-r^{2})r^{4}}{(1-r^{2})^{2}}-\beta=0$$ in the interval $(0,1)$.
Furthermore, the radius $r_{2}$ is the best possible.
\end{theorem}

\begin{proof}

According to the assumption, we have
\begin{align*}
\int\int_{\mathbb{D}_{r}}|h'(z)|^{2}dxdy=\int_{0}^{r}\int_{0}^{2\pi}|h'(Re^{i\theta})|^{2}Rd\theta dR=\pi r^{2}+\pi\sum_{n=2}^{\infty}n|a_{n}|^{2}r^{2n}
\end{align*}
and
\begin{align*}
\int\int_{\mathbb{D}_{r}}|g'(z)|^{2}dxdy=\int_{0}^{r}\int_{0}^{2\pi}|g'(Re^{i\theta})|^{2}Rd\theta dR=\pi\sum_{n=2}^{\infty}n|b_{n}|^{2}r^{2n}.
\end{align*}
The definition of $S_r$ implies that
\begin{align*}
\frac{S_{r}}{\pi}=&\frac{1}{\pi}\int\int_{\mathbb{D}_{r}} J_{f}(z)dxdy=\frac{1}{\pi}\int\int_{\mathbb{D}_{r}}(|h'(z)|^{2}-|g'(z)|^{2})dxdy\\
=&r^{2}+\sum_{n=2}^{\infty}n(|a_{n}|^{2}-|b_{n}|^{2})r^{2n}=r^{2}+\sum_{n=2}^{\infty}n(|a_{n}|+|b_{n}|)(|a_{n}|-|b_{n}|)r^{2n}.\\
\end{align*}
By the inequalities (\ref{3.2}) in the Theorem 3.1, we obtain
\begin{align*}
|a_n|+|b_n|\leq 2(1-\beta) \quad and \quad \left||a_n|-|b_n|\right|\leq 2(1-\beta).
\end{align*}
Then
\begin{align}\label{3.4}
\frac{S_{r}}{\pi}=&r^{2}+\sum_{n=2}^{\infty}n(|a_{n}|+|b_{n}|)(|a_{n}|-|b_{n}|)r^{2n} \notag\\
\leq&r^{2}+4(1-\beta)^{2}\sum_{n=2}^{\infty}nr^{2n}\notag\\
=&r^{2}+4(1-\beta)^{2}\frac{(2-r^{2})r^{4}}{(1-r^{2})^{2}}.
\end{align}
It follows from inequalities (\ref{3.3}) and (\ref{3.4}) that
\begin{align*}
|z|+\sum_{n=2}^{\infty}(|a_{n}|+|b_{n}|)|z|^{n}+\frac{S_{r}}{\pi} \leq r+2(1-\beta)\frac{r^{2}}{1-r}+r^{2}+4(1-\beta)^{2}\frac{(2-r^{2})r^{4}}{(1-r^{2})^{2}}.
\end{align*}
Noting  that $f(0)=0$, and using the Lemma \ref{lemma2.5}, we have
\begin{align*}
 d(f(0),\partial f(\mathbb{D})) =\liminf_{|z|\rightarrow1}|f(z)-f(0)|\geq \beta.
\end{align*}
Now, it is sufficient  to show that $$r+2(1-\beta)\frac{r^{2}}{1-r}+r^{2}+4(1-\beta)^{2}\frac{(2-r^{2})r^{4}}{(1-r^{2})^{2}}\leq\beta$$
holds for $r\leq r_{2}$.

Let $$Q_{2}(r)=r+2(1-\beta)\frac{r^{2}}{1-r}+r^{2}+4(1-\beta)^{2}\frac{(2-r^{2})r^{4}}{(1-r^{2})^{2}}-\beta.$$
Direct computations lead to
\begin{align*}
Q_{2}'(r)=1+2(1-\beta)\frac{(2-r)r}{(1-r)^{2}}+2r+8(1-\beta)^{2}\frac{(4-3r^2+r^{4})r^{3}}{(1-r^{2})^{3}}>0,
\end{align*}
observe that $Q_{2}(0)=-\beta<0$ and $lim_{r\rightarrow1^{-}}Q_{2}(r)>0$, so that there exists a unique number $r_{2}\in (0,1)$ such that $Q_{2}(r_{2})=0$.
Thus, $Q_{2}(r)\leq0$ holds for $r\leq r_{2}$.

Next, to show the radius $r_{2}$ is sharp. Let $$f(z)=z+2(1-\beta)\sum_{n=2}^{\infty}z^{n}.$$
Taking $|z|=r$. According to the above proof, we obtain
\begin{align*}
&|z|+\sum_{n=2}^{\infty}(|a_{k}|+|b_{k}|)|z|^{n}+\frac{S_{r}}{\pi}\\
=&r+2(1-\beta)\frac{r^{2}}{1-r}+r^{2}+4(1-\beta)^{2}\frac{(2-r^{2})r^{4}}{(1-r^{2})^{2}}\\
=&\beta=d(f(0), \partial f(\mathbb{D}))
\end{align*} holds for $|z|=r_{2}$. Thus, the radius $r_{2}$ is sharp.
\end{proof}

\begin{theorem}\label{theorem3.3}
\rm {Suppose $f(z)\in {W^{0}_{H}}(\alpha)$ for $0\leq \alpha<1$, it holds that
\begin{equation*}
|f(z)|+ |z|+\sum_{n=2}^{\infty}(|a_{n}|+|b_n|)|z|^{n} \leq d(f(0),\partial f(\mathbb{D}))
\end{equation*} for $|z|=r\leq r_3$,
where  $r_3$ is the unique root of equation $$ 4\sum_{n=2}^{\infty} \frac{r^n}{\alpha n^2+n(1-\alpha)}+2r-1-2\sum_{n=2}^{\infty} \frac{(-1)^{n-1}}{\alpha n^2+n(1-\alpha)}=0$$ in the interval $(0,1)$.
The radius $r_3$ is the best possible.}
\end{theorem}

\begin{proof}
According to assumption, $f(0)=0$ and Lemma \ref{lemma2.2}, we have
\begin{align*}
d(f(0),\partial f(\mathbb{D}))=\liminf_{|z|\rightarrow1} |f(z)-f(0)|\geq 1+2\sum_{n=2}^{\infty}\frac{(-1)^{n-1}}{\alpha n^2+n(1-\alpha)}.
\end{align*}
It is obvious that $\sum_{n=2}^{\infty}\frac{(-1)^{n-1}}{n(n\alpha+1-\alpha)}$ is convergent and $1+2\sum_{n=2}^{\infty}\frac{(-1)^{n-1}}{n(n\alpha+1-\alpha)}$ for $0\leq \alpha<1$. Then Lemma \ref{lemma2.1} and \ref{lemma2.2} imply that
\begin{align*}
|f(z)|+|z|+\sum_{n=2}^{\infty}(|a_{n}|+|b_n|)|z|^{n}
\leq& r+ 2\sum_{n=2}^{\infty}\frac{r^n}{\alpha n^2+n(1-\alpha)}+r+ 2\sum_{n=2}^{\infty}\frac{r^n}{\alpha n^2+n(1-\alpha)}\\
 =& 2r+ 4\sum_{n=2}^{\infty}\frac{r^n}{\alpha n^2+n(1-\alpha)}.
\end{align*}
Now, it is sufficient to show that $$2r+ 4\sum_{n=2}^{\infty}\frac{r^n}{\alpha n^2+n(1-\alpha)}\leq 1+2\sum_{n=2}^{\infty}\frac{(-1)^{n-1}}{\alpha n^2+n(1-\alpha)}$$ holds for $r\leq r_{3}$. It is equivalent to $Q_{3}(r)\leq0 $ for $r\leq r_{3}$, where
$$ Q_{3}(r)=2r+ 4\sum_{n=2}^{\infty}\frac{r^n}{\alpha n^2+n(1-\alpha)}-1-2\sum_{n=2}^{\infty}\frac{(-1)^{n-1}}{\alpha n^2+n(1-\alpha)}.$$
Observe that
 $$Q_{3}(0)=-1-2\sum_{n=2}^{\infty}\frac{(-1)^{n-1}}{\alpha n^2+n(1-\alpha)}<0, $$
 $$Q_{3}(1)=1+2\sum_{n=2}^{\infty}\frac{1}{\alpha n^2+n(1-\alpha)}+2\sum_{n=1}^{\infty}\frac{1}{2\alpha n^2+n(1-\alpha)} > 0,$$ and
 \begin{align}\label{3.6}
Q_{3}'(r)=2+4\sum_{n=2}^{\infty}\frac{r^{n-1}}{\alpha n+1-\alpha}>0.
\end{align}Thus,  there exists a unique $r_{3} \in (0,1)$ such that $Q_{3}(r_{3}) = 0$ and $Q_{3}(r)\leq0 $
holds for $r\leq r_{3}$.

Next, to show the radius $r_3$ is sharp. Let $$f(z)=z+2\sum_{n=2}^{\infty}\frac{z^{n}}{\alpha n^2+n(1-\alpha)}.$$
Taking $|z|=r$, then
\begin{align*}
|f(z)|+|z|+\sum_{n=2}^{\infty}(|a_{n}|+|b_n|)|z|^{n}
= 2r+ 4\sum_{n=2}^{\infty}\frac{r^n}{\alpha n^2+n(1-\alpha)}.
\end{align*}
Observe that if $r=r_3$, it follows that
$$2r+ 4\sum_{n=2}^{\infty}\frac{r^n}{\alpha n^2+n(1-\alpha)}=1+2\sum_{n=2}^{\infty}\frac{(-1)^{n-1}}{\alpha n^2+n(1-\alpha)}=d(f(0),\partial f(\mathbb{D})) $$
This proves the sharpness of $r_3$ and proof of Theorem 3.3 is complete.

\end{proof}

\begin{theorem}\rm
\label{theorem 3.4} Suppose that $f=h+\overline{g}\in W^{0}_{H}(\alpha)$ for $0\leq \alpha<1$. Let $S_{r}$ denotes the area of the image of the subdisk $|z|<r$ under the mapping $f$. Then
\begin{equation*}
|z|+\sum_{n=2}^{\infty}(|a_{n}|+|b_{n}|)|z|^{n}+\frac{S_{r}}{\pi} \leq d(f(0), \partial f(\mathbb{D}))
\end{equation*} for $|z|=r\leq r_{4}$,
where $r_{4}$ is the unique root of the equation $$r+r^{2}+2\sum_{n=2}^{\infty}\frac{(n\alpha+1-\alpha)r^{n}+2r^{2n}}{n(n\alpha+1-\alpha)^{2}}
-1-2\sum_{n=2}^{\infty}\frac{(-1)^{n-1}}{n(n\alpha+1-\alpha)}=0$$ in  the interval $(0,1)$.
Furthermore, the radius $r_{4}$ is the best possible.
\end{theorem}

\begin{proof}
According to the assumption, $f(0)=0$ and Lemma \ref{lemma2.2}, we have
\begin{align*}
d(f(0),\partial f(\mathbb{D}))=\liminf_{|z|\rightarrow1}|f(z)-f(0)|\geq 1+2\sum_{n=2}^{\infty}\frac{(-1)^{n-1}}{n(n\alpha+1-\alpha)}.
\end{align*}
It is obvious that $\sum_{n=2}^{\infty}\frac{(-1)^{n-1}}{n(n\alpha+1-\alpha)}$ is convergent and $1+2\sum_{n=2}^{\infty}\frac{(-1)^{n-1}}{n(n\alpha+1-\alpha)}$ for $0\leq \alpha<1$. At the same time,  by the Lemma \ref{lemma2.1} $(i)$ and $(ii)$, we have
\begin{align*}
\frac{S_{r}}{\pi}=&\frac{1}{\pi}\int\int_{\mathbb{D}_{r}}(|h'(z)|^{2}-|g'(z)|^{2})dxdy\\
=&r^{2}+\sum_{n=2}^{\infty}n(|a_{n}|^{2}-|b_{n}|^{2})r^{2n}\\
\leq&r^{2}+4\sum_{n=2}^{\infty}\frac{r^{2n}}{n(n\alpha+1-\alpha)^{2}}.
\end{align*}
Thus, using Lemma \ref{lemma2.1} $(i)$ again, we have \begin{align*}
|z|+\sum_{n=2}^{\infty}(|a_{n}|+|b_{n}|)|z|^{n}+ \frac{S_{r}}{\pi} \leq r&+2\sum_{n=2}^{\infty}\frac{r^{n}}{n(n\alpha+1-\alpha)}\\
+&r^{2}+4\sum_{n=2}^{\infty}\frac{r^{2n}}{n(n\alpha+1-\alpha)^{2}}.
\end{align*}
Now, we need to show that $Q_{4}(r)\leq 0$ holds for $r\leq r_{4}$, where
\begin{align*}
  Q_{4}(r)=&r+2\sum_{n=2}^{\infty}\frac{r^{n}}{n(n\alpha+1-\alpha)}+r^{2}+4\sum_{n=2}^{\infty}\frac{r^{2n}}{n(n\alpha+1-\alpha)^{2}}\\
-&1-2\sum_{n=2}^{\infty}\frac{(-1)^{n-1}}{n(n\alpha+1-\alpha)}.
\end{align*}
Direct computations lead to
 $$Q_{4}(0)=-1-2\sum_{n=2}^{\infty}\frac{(-1)^{n-1}}{n(n\alpha+1-\alpha)}<0,$$
$$Q_{4}(1)=1+2\sum_{n=1}^{\infty}\frac{1}{n(2n\alpha+1-\alpha)}+4\sum_{n=2}^{\infty}\frac{1}{n(n\alpha+1-\alpha)^{2}}>0,$$
and
\begin{align*}
Q_{4}'(r)=1+2\sum_{n=2}^{\infty}\frac{r^{n-1}}{n\alpha+1-\alpha}+2r+8\sum_{n=2}^{\infty}\frac{r^{2n-1}}{(n\alpha+1-\alpha)^{2}}>0.
\end{align*}
 So there exists  a unique number $r_{4}\in (0,1)$ such that $Q_{4}(r_{4})=0$.
Thus, $Q_{4}(r)\leq0$ holds for $r\leq r_{4}$.

Next, to show the radius $r_{4}$ is sharp.
Let $$f(z)=z+2\sum_{n=2}^{\infty}\frac{z^{n}}{n(n\alpha+1-\alpha)}, \,\,\,\, |z|=r_{4}.$$
It follows that
\begin{align*}
&|z|+\sum_{n=2}^{\infty}(|a_{n}|+|b_{n}|)|z|^{n}+\frac{S_{r}}{\pi} \\
=&r_{4}+2\sum_{n=2}^{\infty}\frac{r_{4}^{n}}{n(n\alpha+1-\alpha)}+r_{4}^{2}+4\sum_{n=2}^{\infty}\frac{r_{4}^{2n}}{n(n\alpha+1-\alpha)^{2}}\\
=&1+2\sum_{n=2}^{\infty}\frac{(-1)^{n-1}}{n(n\alpha+1-\alpha)}\\
=&d(f(0),\partial f(\mathbb{D})).
\end{align*}
Thus, the radius $r_{4}$ is sharp and the proof of Theorem 3.4 is complete.
\end{proof}

\begin{theorem}\label{theorem3.5}
\rm{Suppose that $f(z) = h + \bar{g}\in \mathcal{G}^{k}_{H}(\alpha)$ with $\alpha\geq \frac{1}{k}$. It holds that
\begin{equation*}
|f(z)|+|z|+ \sum_{n = k+1}^\infty (|a_n|+|b_n|)|z|^n \leq d(f(0),\partial f(\mathbb{D}))
\end{equation*}
for $|z|=r\leq r_{5}$. The radius $r_{5}$ is the unique  root of the equation $$r+\sum_{j = 1}^\infty \frac{r^{kj+1}+(-1)^{j+1}}{1+kj\alpha}+\sum_{n = k+1}^\infty \frac{r^{n}}{1+(n-1)\alpha} -\frac{1}{2}=0$$ in the interval $(0,1)$ and $r_{5}$ is the best possible.}
\end{theorem}

\begin{proof}
 According to the assumption and Lemma \ref{lemma2.4}, we have
\begin{align*}
 d(f(0),\partial f(\mathbb{D}))=\liminf_{|z|\rightarrow1} |f(z)-f(0)|\geq 1+2\sum_{j=1}^{\infty}\frac{(-1)^j}{1+kj\alpha}.
\end{align*}
It is obvious that $\sum_{j=1}^{\infty}\frac{(-1)^j}{1+kj\alpha}$ is convergent and $1+2\sum_{j=1}^{\infty}\frac{(-1)^j}{1+kj\alpha}>0$ for $\alpha\geq  \frac{1}{k}$. Using Lemma \ref{lemma2.3} and \ref{lemma2.4}, we obtain that
\begin{align*}
|f(z)|+|z|+\sum_{n=k+1}^{\infty}(|a_{n}|+|b_n|)|z|^{n}\leq& r+2\sum_{j = 1}^\infty \frac{r^{kj+1}}{1+kj\alpha}+r+2\sum_{n = k+1}^\infty \frac{r^{n}}{1+(n-1)\alpha}\\
=&2r+2\sum_{j = 1}^\infty \frac{r^{kj+1}}{1+kj\alpha}+2\sum_{n = k+1}^\infty \frac{r^{n}}{1+(n-1)\alpha}.
\end{align*}
It is sufficient to show that $$2r+ 2\sum_{j = 1}^\infty \frac{r^{kj+1}}{1+kj\alpha}+2\sum_{n = k+1}^\infty \frac{r^{n}}{1+(n-1)\alpha}\leq 1+2\sum_{j=1}^{\infty}\frac{(-1)^j}{1+kj\alpha}$$ holds for $r\leq r_{5}$. It is equivalent to $Q_{5}(r)\leq0 $ for $r\leq r_{5}$, where
$$ Q_{5}(r)=2r+2\sum_{j = 1}^\infty \frac{r^{kj+1}}{1+kj\alpha}+2\sum_{n = k+1}^\infty \frac{r^{n}}{1+(n-1)\alpha} -1-2\sum_{j=1}^{\infty}\frac{(-1)^j}{1+kj\alpha}.$$
It follows that $$Q'_{5}(r)=2+2\sum_{j = 1}^\infty \frac{(kj+1)r^{kj}}{1+kj\alpha}+2\sum_{n = k+1}^\infty \frac{n r^{n-1}}{1+(n-1)\alpha}>0.$$
We observe that
$$Q_{5}(0)=-1-2\sum_{j = 1}^\infty \frac{(-1)^j}{1+kj\alpha}<0, $$ $$Q_{5}(1)=1+4\sum_{j = 1}^\infty \frac{1}{1+k(2j-1)\alpha}+2\sum_{n = k+1}^\infty \frac{1}{1+(n-1)\alpha}>0.$$
Thus, there exists a unique positive root $r_5\in(0,1)$  such that
$Q_{5}(r_5)=0$. If $r\leq r_5$, it holds
\begin{align*}
|f(z)|+|z|+\sum_{n=k+1}^{\infty}(|a_{n}|+|b_n|)|z|^{n}
\leq& 2r+2\sum_{j = 1}^\infty \frac{r^{kj+1}}{1+kj\alpha}+2\sum_{n = k+1}^\infty \frac{r^{n}}{1+(n-1)\alpha}\\
\leq& 1+2\sum_{j = 1}^\infty \frac{(-1)^j}{1+kj\alpha}.
\end{align*}

Next, we show that the radius $r_5$  is the best possible. For  fixed positive integer $k$, we  consider the function $$f_{k}(z)=z+2\sum\limits_{N\in I} \frac{z^{N}}{1+(N-1)\alpha } \quad I=\left\{N|N=kj+1, j=1,2,\cdot\cdot\cdot\right\}.$$
Taking $|z|=r$ , it follows that
\begin{align*}
|f_{k}(z)|+|z|+\sum_{n=k+1}^{\infty}(|a_{n}|+|b_n|)|z|^{n}=&2r+2\sum_{j = 1}^\infty \frac{r^{kj+1}}{1+kj\alpha}+2\sum\limits_{N\in I} \frac{r^{N}}{1+(N-1)\alpha }.
\end{align*}
According to the above proof, we obtain
 $$2r+2\sum_{j = 1}^\infty \frac{r^{kj+1}}{1+kj\alpha}+2\sum\limits_{N\in I} \frac{r^{N}}{1+(N-1)\alpha }
=1+2\sum_{j = 1}^\infty \frac{(-1)^j}{1+kj\alpha}=d(f_{k}(0),\partial f_{k}(\mathbb{D}))$$ holds for $r=r_5$.
Thus, the proof of Theorem 3.5 is complete.

\end{proof}

\begin{theorem}\rm
\label{theorem3.6} Suppose that $f=h+\overline{g}\in\mathcal{G}^{k}_{H}(\alpha)$ with $\alpha\geq \frac{1}{k}$. Let $S_{r}$ denotes the area of the image of the subdisk $|z|<r$ under the mapping $f$. Then
\begin{equation}
|z|+\sum_{n=k+1}^{\infty}(|a_{n}|+|b_{n}|)|z|^{n}+\frac{S_{r}}{\pi} \leq d(f(0),\partial f(\mathbb{D}))
\end{equation} for $|z|=r\leq r_{6}$,
where $r_{6}$ is the unique  root of the equation $$r+2\sum_{n=k+1}^{\infty}\frac{(n\alpha+1-\alpha)r^{n}+2nr^{2n}}{(n\alpha+1-\alpha)^2}+r^{2}
-1-2\sum_{j=1}^{\infty}\frac{(-1)^{j}}{{1+kj\alpha}}=0$$ in the interval $(0,1)$ and the radius $r_{6}$ is the best possible.
\end{theorem}

\begin{proof}
Lemma \ref{lemma2.4} implies that
\begin{align*}
d(f(0),\partial f(\mathbb{D}))=\liminf_{|z|\rightarrow1}|f(z)-f(0)|\geq1+2\sum_{j=1}^{\infty}\frac{(-1)^{j}}{{1+kj\alpha}}.
\end{align*}
It is obvious that $\sum_{j=1}^{\infty}\frac{(-1)^j}{1+kj\alpha}$ is convergent and $1+2\sum_{j=1}^{\infty}\frac{(-1)^j}{1+kj\alpha}>0$ for $\alpha\geq   \frac{1}{k}$. Lemma \ref{lemma2.3} and \ref{lemma2.4} lead to that
\begin{align*}
\frac{S_{r}}{\pi}=&\frac{1}{\pi}\int\int_{\mathbb{D}_{r}
}(|h'(z)|^{2}-|g'(z)|^{2})dxdy\\
=&r^{2}+\sum_{n=k+1}^{\infty}n(|a_{n}|^{2}-|b_{n}|^{2})r^{2n}\\
\leq& r^{2}+4\sum_{n=k+1}^{\infty}\frac{nr^{2n}}{(n\alpha+1-\alpha)^{2}}.
\end{align*}
It follows that \begin{align*}
|z|+\sum_{n=k+1}^{\infty}(|a_{n}|+|b_{n}|)|z|^{n}+\frac{S_{r}}{\pi}
&\leq r+2\sum_{n=k+1}^{\infty}\frac{r^{n}}{n\alpha+1-\alpha}\\
&+r^{2}+4\sum_{n=k+1}^{\infty}\frac{nr^{2n}}{(n\alpha+1-\alpha)^{2}}.
\end{align*}
Now, we need to show that $ Q_{6}(r)\leq 0 $ holds for $r\leq r_{6}$, where
\begin{align*}
 Q_{6}(r)=&r+2\sum_{n=k+1}^{\infty}\frac{r^{n}}{n\alpha+1-\alpha}+r^{2}+4\sum_{n=k+1}^{\infty}\frac{nr^{2n}}{(n\alpha+1-\alpha)^{2}}
-1-2\sum_{j=1}^{\infty}\frac{(-1)^{j}}{{1+kj\alpha}}.
\end{align*}
Direct computations lead to
$$ Q_{6}(0)=-1-2\sum_{j=1}^{\infty}\frac{(-1)^{j}}{{1+kj\alpha}}<0,$$
$$ Q_{6}(1)=1+2\sum_{n=k+1}^{\infty}\frac{1}{n\alpha+1-\alpha}
+4\sum_{n=k+1}^{\infty}\frac{n}{(n\alpha+1-\alpha)^{2}}-2\sum_{j=1}^{\infty}\frac{(-1)^{j}}{{1+kj\alpha}}>0,$$and
\begin{align*}
 Q_{6}'(r)=1+2\sum_{n=k+1}^{\infty}\frac{nr^{n-1}}{n\alpha+1-\alpha}+2r+8\sum_{n=k+1}^{\infty}\frac{n^2r^{2n-1}}{(n\alpha+1-\alpha)^{2}}>0.
\end{align*}
So, there exists a unique root $r_{6}\in (0,1)$ such that $ Q_{6}(r_{6})=0$.
Thus, $ Q_{6}(r)\leq 0$ holds for $r\leq r_{6}$.

Next, to show the radius $r_{6}$ is the best possible. For  fixed positive integer $k$, let
$$f_{k}(z)=z+2\sum\limits_{N\in I} \frac{z^{N}}{1+(N-1)\alpha } \quad I=\left\{N|N=kj+1, j=1,2,\cdot\cdot\cdot\right\}.$$
Taking $|z|=r_{6}$. According to the above proof, we obtain
\begin{align*}
&|z|+\sum_{n=k+1}^{\infty}(|a_{n}|+|b_{n}|)|z|^{n}+\frac{S_{r}}{\pi}\\
=&r_{6}+2\sum_{N\in I}\frac{r_{6}^{N}}{1+(N-1)\alpha }+r_{6}^{2}+4\sum_{N\in I}\frac{Nr_{6}^{2N}}{(1+(N-1)\alpha )^{2}}\\
=&1+2\sum_{j=1}^{\infty}\frac{(-1)^{j}}{{1+kj\alpha}}=d(f_{k}(0),\partial f_{k}(\mathbb{D})).
\end{align*}
Thus, the proof of Theorem 3.6 is complete.\\
\end{proof}

\medskip

{\bf Declarations}\\

{\bf Data availability }\quad  The authors declare that this research is purely theoretical and does not associate with any
data.\\

{\bf Conflict of interest}  \quad The authors declare that they have no conflict of interest, regarding the publication of
this paper.


\begin{thebibliography}{10}

\bibitem{abu2017bohr}
Y.~Abu-Muhanna, R.~M. Ali, and Z.~C. Ng.
\newblock Bohr radius for the punctured disk.
\newblock {\em Math. Nachr.}, 290(16): 2434--2443, 2017.

\bibitem{ali2016bohr}
R.~M. Ali, Z.~Abdulhadi, and Z.~C. Ng.
\newblock The bohr radius for starlike logharmonic mappings.
\newblock {\em Complex Var. Elliptic Equ.}, 61(1): 1--14, 2016.

\bibitem{ali2017note}
R.~M. Ali, R.~W. Barnard, and A.~Y. Solynin.
\newblock A note on bohr's phenomenon for power series.
\newblock {\em J. Math. Anal. Appl.}, 449(1): 154--167, 2017.

\bibitem{allu2021bohr}
V.~Allu and H.~Halder.
\newblock Bohr phenomenon for certain subclasses of harmonic mappings.
\newblock {\em Bull. Sci. Math.}, 173: 103053, 2021.

\bibitem{allu2021starlike}
V.~Allu and H.~Halder.
\newblock Bohr radius for certain classes of starlike and convex univalent
  functions.
\newblock {\em J. Math. Anal. Appl.}, 493(1): 124519, 2021.

\bibitem{allu2021halder}
V.~Allu and H.~Halder.
\newblock Bohr phenomenon for certain close-to-convex analytic functions.
\newblock {\em Comput. Methods Funct. Theory}, 22(3): 491--517, 2022.

\bibitem{axler2013harmonic}
S.~Axler, P.~Bourdon, and R.~Wade.
\newblock {\em Harmonic function theory}, 137.
\newblock Springer Science \& Business Media, 2013.

\bibitem{bohr1914theorem}
H.~Bohr.
\newblock A theorem concerning power series.
\newblock {\em Proc. Lond. Math. Soc.}, 2(1): 1--5, 1914.

\bibitem{chuaqui2003schwarzian}
M.~Chuaqui, P.~Duren, and B.~Osgood.
\newblock The schwarzian derivative for harmonic mappings.
\newblock {\em J. Anal. Math.}, 91(1): 329--351, 2003.

\bibitem{evdoridis2021improved}
S.~Evdoridis, S.~Ponnusamy, and A.~Rasila.
\newblock Improved bohr's inequality for shifted disks.
\newblock {\em Results Math.}, 76(1): 14, 2021.

\bibitem{gangania2022bohr}
K.~Gangania and S.~S. Kumar.
\newblock Bohr radius for some classes of harmonic mappings.
\newblock {\em Iran. J. Sci. Technol. Trans. A Sci.}, 46(3): 883--890, 2022.

\bibitem{ghosh2019subclass}
N.~Ghosh and A.~Vasudevarao.
\newblock On a subclass of harmonic close-to-convex mappings.
\newblock {\em Monatsh. Math.}, 188: 247--267, 2019.

\bibitem{hu2022bohr}
X.-J. Hu, Q.-H. Wang, and B.-Y. Long.
\newblock Bohr-type inequalities with one parameter for bounded analytic functions of Schwarz functions.
\newblock {\em Bull. Malays. Math. Sci. Soc.}, 45: 575-591, 2022.

\bibitem{huang2021bohr}
Y.~Huang, M.-S. Liu, and S.~Ponnusamy.
\newblock Bohr-type inequalities for harmonic mappings with a multiple zero at
  the origin.
\newblock {\em Mediterr. J. Math.}, 18: 1--22, 2021.

\bibitem{iwaniec2011nitsche}
T.~Iwaniec, L.V. Kovalev, and J.~Onninen.
\newblock The nitsche conjecture.
\newblock {\em J. Amer. Math. Soc.}, 24(2): 345--373, 2011.

\bibitem{iwaniec2019rado}
T.~Iwaniec and J.~Onninen.
\newblock Rad\'{o}-kneser-choquet theorem for simply connected domains
  (p-harmonic setting).
\newblock {\em Trans. Amer. Math. Soc.}, 371(4): 2307--2341, 2019.

\bibitem{kayumov2017bohr}
I.~R. Kayumov and S.~Ponnusamy.
\newblock Bohr--rogosinski radius for analytic functions.
\newblock {\em arXiv preprint arXiv: 1708.05585}, 2017.

\bibitem{kayumov2018improved}
I.~R. Kayumov and S.~Ponnusamy.
\newblock Improved version of bohr's inequality.
\newblock {\em C. R. Math. Acad. Sci. Paris}, 356(3): 272--277, 2018.

\bibitem{liu2020bohr}
M.-S. Liu, S.~Ponnusamy, and J.~Wang.
\newblock Bohr's phenomenon for the classes of quasi-subordination and
  k-quasiregular harmonic mappings.
\newblock {\em Rev. R. Acad. Cienc. Exactas F\'{i}s. Nat. Ser. A Mat. RACSAM},
  114: 1--15, 2020.

\bibitem{Ming2018Bohr}
M.-S. Liu, Y.-M. Shang, and J.-F. Xu.
\newblock Bohr-type inequalities of analytic functions.
\newblock {\em J. Inequal. Appl.}, 2018.

\bibitem{liu2019geometric}
M.-S. Liu and L.-M. Yang.
\newblock Geometric properties and sections for certain subclasses of harmonic
  mappings.
\newblock {\em Monatsh. Math.}, 190: 353--387, 2019.

\bibitem{liu2019bohr}
Z.-H. Liu and S.~Ponnusamy.
\newblock Bohr radius for subordination and k-quasiconformal harmonic mappings.
\newblock {\em Bull. Malays. Math. Sci. Soc.}, 42: 2151--2168, 2019.

\bibitem{schoen1978univalent}
R.~Schoen and S.~Yau.
\newblock On univalent harmonic maps between surfaces.
\newblock {\em Invent. math.}, 44(3): 265--278, 1978.

\bibitem{wule2022}
L.~Wu, Q.-H. Wang, and B.-Y. Long.
\newblock Some bohr-type inequalities with one parameter for bounded analytic
  functions.
\newblock {\em Rev. R. Acad. Cienc. Exactas F\'{i}s. Nat. Ser. A Mat. RACSAM},
  116(2): 1--13, 2022.

\end{thebibliography}
\end{document}